\documentclass[12pt]{article}

\usepackage[margin=1.20in]{geometry}

\usepackage{amssymb,amsmath}
\usepackage{amsthm,enumitem}
\usepackage{mathrsfs}
\usepackage{hyperref}
\usepackage{textcomp}
\hypersetup{
    colorlinks,%
    citecolor=black,%
    filecolor=black,%
    linkcolor=black,%
    urlcolor=black
}
\usepackage[abbrev,nobysame]{amsrefs}

\usepackage{verbatim}

\usepackage{sectsty}
\usepackage[nocompress]{cite}

\makeatletter

\newdimen\bibspace
\makeatother

\makeatletter

\newtheorem{Theorem}{Theorem}[section]
\newtheorem{Lemma}[Theorem]{Lemma}
\newtheorem{Proposition}[Theorem]{Proposition}
\newtheorem*{Assumption*}{Assumption (H)}
\newtheorem{Definition}[Theorem]{Definition}

\newtheorem{Remark}[Theorem]{Remark}

\newtheorem{question}[Theorem]{Question}

\def\XXint#1#2#3{{\setbox0=\hbox{$#1{#2#3}{\int}$}
  \vcenter{\hbox{$#2#3$}}\kern-.5\wd0}}

\newcommand{\be}{\begin{equation}}      \newcommand{\ee}{\end{equation}}

\newcommand{\T}{\mathcal{T}}
\newcommand{\R}{\mathbb{R}}              

\newcommand{\D}{\mathcal{D}}
\newcommand{\I}{\mathcal{I}}

\begin{document}

\title{\textbf{On the singular set of the free boundary for a Monge-Amp\`ere obstacle problem}\bigskip}

\author{\medskip  Tianling Jin,\footnote{T. Jin was partially supported by NSFC grant 12122120, and Hong Kong RGC grants GRF 16304125, GRF 16303624 and GRF 16303822.}\quad Xushan Tu, \quad
Jingang Xiong\footnote{J. Xiong was partially supported by NSFC grants 12325104 and 12271028.}}

\date{\today}

\maketitle

\begin{abstract}  
This is a continuation of our earlier work \cite{jin2025regularity} on the Monge-Amp\`ere obstacle problem
\[
\det D^2 v = v^q \chi_{\{v>0\}}, \quad v \geq 0 \text{ convex} 
\] 
with $q \in [0,n)$, where we studied the regularity of the strictly convex part of the free boundary.

In this work, we examine the non-strictly convex part of the free boundary and
establish optimal dimension bounds for its flat portion. Additionally, we investigate    
the strong maximum principle and a stability property for this Monge-Amp\`ere obstacle problem.
\medskip


\medskip


\end{abstract}

\section{Introduction}

Let $\Omega \subset \R^n$ be a bounded convex open set and let $g$ be a positive bounded function on $\Omega$. The classical Monge-Amp\`ere equation 
\begin{equation}\label{eq:mae classical}
\det D^2 w = g \quad \text{in } \Omega
\end{equation} 
concerns convex Aleksandrov solutions $w\colon \Omega \to \R$. 
The convexity assumption of $w$ guarantees the (possibly degenerate) ellipticity of the equation. The notion of strict convexity plays a fundamental role in the study of these solutions. Its importance is exemplified by Caffarelli's celebrated result \cite{caffarelli1990ilocalization,caffarelli1990interiorw2p,caffarelli1991regularity} which embody the ``strict convexity implies regularity'' principle -- a cornerstone in the field. The strict convexity is not merely a technical assumption; rather, it is actually necessary for the regularity of solutions by the famous Pogorelov examples;  see  the books \cite{figalli2017monge, gutierrez2016monge,trudinger2008monge,savin2024monge}. 

For a solution $w$ of \eqref{eq:mae classical}, let us denote by $\Sigma_w$ its ``singular set" consisting of non-strictly convex points of $w$, that is,
\[
\Sigma_w:=\{ x \in \Omega:\; \exists z\in \Omega \setminus \{x\} \text{ and } p \in \partial w (x)  \quad s.t. \quad w(z)=w(x)+p\cdot(z-x) \},
\]
or equivalently 
\[
\Sigma_w=\{ x \in \Omega:\; x \text{ is not an exposed point
of the epigraph of } w\} .
\]
From Caffarelli's papers \cite{caffarelli1990ilocalization,caffarelli1993note}, it is known that $\Sigma_w$
is a closed (relatively to $\Omega$) set that emanates from $\partial \Omega$. More precisely, any affine component of $w$ -- that is,  a maximal non-trivial convex subset $E \subset \Omega$ on which $w$ is linear -- satisfies $(\overline{E})^{ext} \subset \partial\Omega$ with the dimensional constraint $\dim E < n/2$, where $(\overline{E})^{ext}$ denotes the set of extreme points of $\overline{E}$. 

In \cite{savin2005obstacle}, Savin studied convex solutions to the Monge-Amp\`ere obstacle problem
\begin{equation}\label{eq:obs problem savin}
\det  D^2 v = g \chi_{\{v > 0\}} \ \text{ in } \Omega, \quad  v> 0\ \text{ on } \partial\Omega,
\end{equation}
where $\chi_E$ denotes the characteristic function of a set $E$, and $g$ is a positive bounded function satisfying $0 < \lambda \leq g \leq \Lambda$ on $\Omega$. He proved that the free boundary $\partial \{v=0\}$ is $C^{1,\alpha}$, and furthermore, when $g\equiv 1$, the free boundary is $C^{1,1}$ and  uniformly  convex. G\'alvez-Jim\'enez-Mira \cite{galvez2015classification} proved in two dimensions and Huang-Tang-Wang \cite{huang2024regularity} in higher dimensions that the free boundary is $C^{\infty}$  and analytic, provided that $g$ is $C^{\infty}$ and analytic, respectively. 

Let $q  \in [0, n)$ be a fixed constant. In \cite{jin2025regularity}, we studied the following Monge-Amp\`ere obstacle problem
\begin{equation}\label{eq:obs problem q}
\det D^2 v = g v^{q} \chi_{\{v > 0\}} \quad \text{in } \Omega, \quad v \geq 0 \text{ is convex}, 
\end{equation}
which emerges from the $L_p$ Minkowski problem (see Section 6 in \cite{jin2025regularity}). 
Let the convex set 
\[
K:= \left\{ x\in\Omega:\; v(x)=0\right\} 
\]
be the coincidence set, and $$\Gamma= \Omega \cap \partial K$$ be the free boundary in $\Omega$.
In \cite{jin2025regularity}, we also verified the ``strict convexity implies regularity" principle, i.e., we proved local smoothness for both solutions and free boundaries near the strictly convex part $\Gamma_{sc}$ of the free boundary, where
\[
\Gamma_{sc}:= \left\{x \in \Omega:\; x \text{ is an exposed point of } K  \right\}\subset \Gamma.
\]
Since we allow $v$ to vanish somewhere on $\partial\Omega$, it is possible that $\Gamma_{sc}\neq \Gamma$ and the solution $v$ can be merely Lipschitz at 
\[
\Gamma_{nsc}:=\Gamma\setminus\Gamma_{sc};
\] 
see the examples in Remark 2.19 in \cite{jin2025regularity}. Therefore, in this framework, the non-strictly convex component $\Gamma_{nsc}$ emerges naturally as the free boundary's singular set, mirroring the role of $\Sigma_w$ in the classical theory.

In this paper, we extend Caffarelli's optimal dimension estimate in \cite{caffarelli1993note} for the Monge-Amp\`ere equation  \eqref{eq:mae classical} to the Monge-Amp\`ere obstacle problem \eqref{eq:obs problem q}  for the flat parts in $\Gamma_{nsc}$. 

\begin{Theorem}[Dimension estimate for flat parts in $\Gamma_{nsc}$]\label{thm:dim of E} 
Let $q\in [0,n)$, and let $v\not\equiv 0, v \in C(\overline{\Omega})$ be a convex non-negative function satisfying
\begin{equation}\label{eq:obs subsoultion}
\det D^2 v \geq  v^{q} \chi_{\{v > 0\}} \quad \text{in } \Omega. 
\end{equation}
For any convex subset $E \subset \Gamma_{nsc}$, there holds
\begin{equation}\label{eq:dim estimate lip}
 \dim E < \frac{n + q}{2}.
\end{equation}
Furthermore, if there exists a point $x_0 \in \overline{E} \cap \partial\Omega$ and constants $C, s > 1$ such that 
\begin{equation}\label{eq:pointwise regular on bd}
v(x) \leq C|x-x_0|^{s} ,\quad \forall x \in \partial \Omega,
\end{equation}
then
\begin{equation}\label{eq:dim estimate c1}
\dim E  \leq n-\frac{(n-q)s}{2}.
\end{equation}
\end{Theorem}

The estimates \eqref{eq:dim estimate lip} and \eqref{eq:dim estimate c1} in Theorem \ref{thm:dim of E} are optimal, as shown in Theorem \ref{thm:dim optimize2} below.
\begin{Theorem}\label{thm:dim optimize2} 
Suppose $n+q>2$. For every positive integer $k < \frac{n+q}{2}$ and real $s \in [1, \frac{2n-2k}{n-q}]$, there exists a solution to  
\begin{equation}\label{eq:obs eq q g=1}
\det  D^2 v = v^{q} \chi_{\{v > 0\}} \quad 
\text{in } B_1(0), \quad v \geq 0 \text{ is convex}, 
\end{equation} 
 that
satisfies $K=\R^k\cap B_1(0)$  and 
\[
c\operatorname{dist} (x,\R^k)^{s} \leq v(x)  \leq C\operatorname{dist} (x,\R^k)^{s}\quad \text{in } B_1(0)
\]
holds for some positive constants $c$ and $C$. 
\end{Theorem}
Theorem \ref{thm:dim optimize2} is inspired by the works of Mooney \cite{mooney2015singular}, and Caffarelli and Yuan \cite{caffarelli2022singular}. When $q=0$, it recovers the result of Caffarelli and Yuan \cite{caffarelli2022singular}.

Note that  $|K| = 0$ in Theorem \ref{thm:dim optimize2}, where $|K|$ denotes the $n$-dimensional Lebesgue measure of $K$. We also have the following example, in which $|K|>0$,  showing the optimality of \eqref{eq:dim estimate lip}. This construction is motivated by the work of Mooney and Rakshit \cite{mooney2021solutions, mooney2023singular}.

 \begin{Theorem}\label{thm:dim optimize} 
Suppose $n+q>2$. There exist non-trivial solutions to \eqref{eq:obs eq q g=1} such that $|K|>0$ and $\Gamma_{nsc}$ contains a convex subset $E$ of dimension $\left\lceil\frac{n+q}{2}\right\rceil - 1$,  where $\left\lceil \frac{n+q}{2} \right\rceil$ denotes the smallest integer larger than $\frac{n+q}{2}$.
\end{Theorem}


\begin{Remark}
Our examples will also show the following phenomena:
\begin{itemize}
\item If $n + q > 2$, there exists a solution of \eqref{eq:obs eq q g=1} that is merely Lipschitz continuous. (Note that when $n=2$ and $q=0$, it was proved in  Proposition 2.18 in \cite{jin2025regularity} that $\Gamma_{nsc}= \emptyset$, and thus, the solution is $C^{1,\alpha}$ regular.)
\item If $q > 0$, there exists a merely Lipschitz solution of \eqref{eq:obs eq q g=1} for which its free boundary $\Gamma = \Gamma_{nsc}$ is smooth and has Hausdorff dimension $n - 1$. 
\item If $n \geq 3$, there exists a merely Lipschitz solution of \eqref{eq:obs eq q g=1} for which its free boundary $\Gamma$ is merely Lipschitz continuous. 

\end{itemize}

\end{Remark}

In the end of the paper, we investigate two properties for the Monge-Amp\`ere obstacle problem \eqref{eq:obs problem q}:
\begin{itemize}
\item[(i).] We discuss the validity of \emph{the strong maximum principle} near $\Gamma_{sc}$ in Proposition \ref{prop:strict decay} and show its failure at $\Gamma_{nsc}$ in Proposition \ref{prop:non-strict decay};
\item[(ii).]  
We discuss the \emph{stability} of the coincidence set under the uniform convergence of solutions in Proposition \ref{prop:stability of K}.
\end{itemize}

The paper is organized as follows: In Section \ref{sec:dimension estimates}, we prove the dimension estimates in Theorem \ref{thm:dim of E} and demonstrate their optimality as stated in Theorems \ref{thm:dim optimize2} and \ref{thm:dim optimize}. We also provide examples of an $(n-1)$-dimensional singular set of the free boundary and the failure of $W^{2,1}$ estimates in Proposition \ref{prop:cylinder K}. In Section \ref{sec:smp}, we present two additional properties -- the strong maximum principle and stability -- for the Monge-Amp\`ere obstacle problem \eqref{eq:obs problem q}.

\section{Dimension estimates for flat portions and examples}\label{sec:dimension estimates}

Let us first recall some  results on the Monge-Amp\`ere obstacle problem \eqref{eq:obs problem q}  which were obtained in our earlier work \cite{{jin2025regularity}}.

\begin{Lemma}[Comparison Principle, Lemma 2.3 in \cite{jin2025regularity}]\label{lem:comparison principle obs} 
Suppose the nonnegative convex functions $\underline w \in C(\overline{ \Omega})$ and $\overline{w} \in C(\overline{ \Omega})$ satisfy 
\[
\det D^2 \underline{w} \geq  g\underline{w}^q \chi_{\{\underline w>0\}}  \quad  \text{in }\Omega, \quad 
\det D^2 \overline w\le  g\overline{w}^q     \quad  \text{in }\Omega,
\]
and $\overline w\geq \underline w$ on $\partial \Omega$, then  $\overline w \geq \underline w$ in $ \Omega$.
\end{Lemma}

The existence and uniqueness of solutions to \eqref{eq:obs problem q}  were established via Perron's method.  

\begin{Lemma}[Existence and uniqueness, Proposition 2.2 and Corollary 2.4 in \cite{jin2025regularity}]\label{lem:existence and uniqueness} 
Let $\varphi \in C(\overline{\Omega})$ be a convex function with $\varphi \geq 0$ on $\partial\Omega$. Then there exists a unique convex non-negative solution to the Dirichlet problem 
\[
\det  D^2 v = g v^{q} \chi_{\{v > 0\}} \quad 
\text{in }   \Omega,\quad 
v=\varphi \quad \text{on } \partial \Omega.
\]
\end{Lemma}

We will also use the lemma below, which follows from the comparison principle.

\begin{Lemma}[Lemma 2.4 in \cite{jin2025regularity}]\label{lem:volume concidence}
Let $h>0$ be a constant. Suppose the convex function $ w>0$ satisfies \eqref{eq:obs subsoultion} on a convex closed set $O$ with $w \leq h $ on $\partial O$. Then
\[
|O| \leq C(n,q)h^{\frac{n-q}{2}} .
\]
\end{Lemma}

The structure on $\Gamma_{nsc}$ can be described as follows. 

\begin{Proposition}[Proposition 3.1, Theorem 1.4, Theorem 1.7 in \cite{jin2025regularity}]\label{thm:nsc set}	
Suppose $v$ is a solution to \eqref{eq:obs problem q}. Define $\Sigma_v$ as the union of all convex sets $E\subset\Omega$ such that
\begin{itemize}
\item $v$ is linear on $E$;
\item $(\overline E)^{ext} \subset \partial \Omega$.
\end{itemize}
Then,  
\[
\Gamma_{nsc}= \Sigma_v \cap \Gamma,
\]  
$\Sigma_v$ is relatively closed to $\Omega$, $v \in C^{1,\beta}(\Omega\setminus \Sigma_{v})$ and is strictly convex in $\Omega \setminus (\Sigma_v\cup K)$, where $\beta \in (0,1)$ depends only on $n$, $q$, $\lambda$, and $\Lambda$. 

\end{Proposition}

\begin{Remark}\label{rem:dichotomy for Ktype}
Outside $K$, the set $\Sigma_v$ is the set of non-strictly convex points of $v$. Based on Proposition \ref{thm:nsc set}, we can further characterize the free boundary $\Gamma_{nsc}:=\Gamma \setminus \Gamma_{sc}  $ as follows:\footnote{For the sake of simplicity, we define an exposed face as the intersection of  
$K$ with one of its supporting hyperplane in $\R^n$. Additionally, we regard $K$ itself as an exposed face of $K$ when $|K| = 0$.} 
\begin{equation}\label{eq:definition of gammansc}
\Gamma_{nsc} :=  \bigcup \{E : E  \text{ is a non-trivial exposed face of } K \}.
\end{equation}

The study of $v$ and $\partial K$ is usually divided into two cases: 
\begin{itemize}
\item[(i)] $|K|>0$; 
\item[(ii)] $|K| = 0$, from which it follows that either $K=\Gamma_{nsc}$ or $K$ is a singleton. 
\end{itemize}
\end{Remark}

\subsection{Dimension estimates in Theorem \ref{thm:dim of E}}
\label{sec:dim_analysis}
\begin{proof}[Proof of Theorem \ref{thm:dim of E}.]
For simplicity, let us assume that the mass center of $E$ is at $0$, and choose $p \in \partial v(0)$ that achieves 
\[
\sup \left\{ |p | :\; p \in \partial v(0) \right\}.
\]
After an affine transformation, we may further assume that  $p=|p|e_1$, $K\subset \left\{x_1 \leq  0\right\}  $,  and $E \subset \left\{ x_1=0\right\}$. 
Then, we shall consider the function $\tilde{v}:=v -|p| x_1$.

For every sufficiently small $\varepsilon>0$,  in the domain $\{x_1 > 0\}$, the function $\tilde{v}_{\varepsilon} := v - (1-\varepsilon)|p|x_1$ is positive and satisfies:
\[
\det D^2 \tilde{v}_{\varepsilon}= \det D^2  v =v^q\chi_{\{ v > 0\}}  \geq  \tilde{v}_{\varepsilon}^q.
\]
Lemma \ref{lem:volume concidence} now implies that
\[
|S_h^{\tilde{v}_{\varepsilon}} \cap \{ x_1 > 0 \}| \leq C(n, q) h^{\frac{n - q}{2}}, \quad \forall h>0,
\]
where $S_h^{\tilde{v}_{\varepsilon}} := \{x:\;  \tilde{v}_{\varepsilon}(x) < h \}$.
Sending $\varepsilon \to 0$, the corresponding set $S_h^{\tilde{v}} := \{x:\; \tilde{v}(x) < h \}$ then satisfies
\[
|S_h^{\tilde{v}} \cap \{ x_1 > 0 \}| \leq C(n, q) h^{\frac{n - q}{2}}, \quad \forall h>0.
\]
Note that we did not apply Lemma \ref{lem:volume concidence} directly to $\tilde v$  since $\tilde v$ may not be positive in $\{ x_1 > 0 \}$.

For small $h > 0$, the Lipschitz regularity of $\tilde{v}$ implies that the convex hull
\[
\operatorname{conv}\left\{E, c_{v}B_{h}(0) \right\} \subset S_h^{\tilde{v}}
\]
for some $c_v>0$ independent of $h$.
By the definitions of $\tilde{v}$ and $K$, $\tilde{v}$ is $C^1$ along the direction of $e_1$ at $0$. Suppose $\dim  (E\cap \Omega)=k$, then we have
\[
\omega(h) h^{n-k-1} |E\cap \Omega|_{\mathcal{H}^k}   \leq    |S_h^{\tilde{v}} \cap \{ x_1 > 0 \}| \leq C(n,q)   h^{\frac{n - q}{2}} ,
\]
where $\omega(h)$ is a quantity satisfying $\frac{\omega(h)}{h} \to +\infty$ as $h\to 0$ and $\mathcal{H}^k$ denotes the $k$-dimensional Hausdorff measure.
Sending $h \to 0$, we obtain \eqref{eq:dim estimate lip}.

Suppose in addition that \eqref{eq:pointwise regular on bd} holds at some $x_0 \in \overline{E} \cap \partial \Omega$. Since the ray $\{ (1-t)x_0  :\;t > 0\}$ is oblique at $x_0$, we obtain the improved estimate:
\[
S_h^{\tilde{v}}\supset \operatorname{conv}\left\{E, c_{v,E}B_{h^{\frac{1}{s}}}(0) \cap\partial\Omega  \right\}
\]
for some $c_{v,E}>0$ independent of $h$. Therefore,
\[
   (c_{v, E} h^{\frac{1}{s}})^{n-k} |E\cap \Omega|_{\mathcal{H}^k} \leq  |S_h^{\tilde{v}} \cap \left\{ x_1>0\right\}|\leq  C(n,q)h^{\frac{n-q}{2}}. 
\]
By letting $h \to 0$, we obtain $\frac{n-k}{s} \geq \frac{n-q}{2}$, which yields \eqref{eq:dim estimate c1}.
\end{proof}

\begin{Remark}
If  $v$ is a solution to \eqref{eq:obs problem q}, then Proposition \ref{thm:nsc set} implies that $\Gamma_{nsc}$ emanates from $\partial\Omega$.
Consequently, in the spirit of Corollary 4 in \cite{caffarelli1990ilocalization}, Theorem \ref{thm:dim of E} further establishes that if $s>\frac{2(n-1)}{n-q}$, then $\Gamma_{nsc} = \emptyset$.  
\end{Remark}

\subsection{Examples}\label{sec:examples}

We recall the equation (5.5.2) in \cite{gutierrez2016monge}:

\begin{Lemma}[Equation (5.5.2) in \cite{gutierrez2016monge}]\label{lem:determinat formula}
Let $n\ge 3$, $x\in\R^n$  and set $x=(y,z)$; $y=(y_1,\cdots ,y_{n-k})$; $z=(z_{1},\cdots ,z_{k})$; $\rho=|y|=(\sum_{i=1}^{n-k}y_i^2)^{1/2}$; $r=|z|=(\sum_{i=1}^{k}z_i^2)^{1/2}$; $u(\rho,r)$ a function of two variables and
\[
w(x)=w(y,z)=u(|y|,|z|) .
\]
We have 
\[
\det D^2 w =\left(\frac{u_{\rho}}{\rho}\right)^{n-k-1}\left(\frac{u_{r}}{r}\right)^{k-1}(u_{\rho \rho}u_{rr}-u_{\rho r}^2).
\]
\end{Lemma}

\begin{proof}[Proof of Theorem \ref{thm:dim optimize2}]
Let us consider the coordinates in Lemma \ref{lem:determinat formula}. Set $f(r)=\left(1+\frac{1}{2}r^2\right)$, and 
\[
w\left(y, z\right)= u(|y|,|z|).
\]
To streamline our analysis, we will construct solutions through appropriately chosen subsolutions. The explicit construction in the case $q=0$ can be found in \cite{mooney2015singular} and \cite{caffarelli2022singular}.

{\bf (i): The case of $s=1$.}  Let us consider
\[
  u(\rho,r)=\rho+\rho^{\beta}f(r),\quad \beta=\frac{n-k+1+q}{k+1}.
\] 
The assumption $k <\frac{n+q}{2}$ ensures that $\beta >1$. Since the Hessian $D^2 w$ at the point $x=\rho e_{n-k}+re_{n}$ takes the form
\[
\left[\begin{array}{cccc}
\frac{1+\beta \rho^{\beta-1} f}{\rho}  \I_{n-k-1} & 0 & 0 & 0\\
0 & \beta(\beta -1)\rho^{\beta-2}f  & 0 & \beta \rho^{\beta-1} r\\
0 & 0 & \rho^{\beta} \I_{k-1}  & 0\\
0 & \beta \rho^{\beta-1}r  & 0 & \rho^{\beta}
\end{array}\right],
\]
then $w$ is a convex function in $\left\{ r \leq \tau\right\}$ if $\tau>0$ is sufficiently small, and satisfies (see Lemma \ref{lem:determinat formula})
\[
\det D^2 w
=\rho^q\left(1+\beta \rho^{\beta-1} f\right)^{n-k-1}
\left[\beta(\beta -1)f - \beta^2 r^2\right] \ge cw^q   \quad \text{in } B_\tau(0) 
\]
in the Aleksandrov sense for some positive constant $c$. Let $v \geq 0 $ be a solution of
\[
\det  D^2 v=v^q \chi_{\{v>0\}}\quad \text{in } B_\tau(0), \quad v = c^{\frac{1}{q-n}}w \quad \text{on } \partial B_\tau (0).
\]
By combining the comparison principle with the convexity of $v$, it follows that 
\[
c^{\frac{1}{q-n}}|y| \leq c^{\frac{1}{q-n}}w(y,z) \leq v\left(y, z\right) \leq C|y|.
\]
It is straightforward to verify that   $\tau^{-\frac{2n}{n-q}}v(\tau x)$ meets our requirements.

{\bf (ii): The case of $s >1$.} Let us consider
\[
u\left(\rho, r\right)=\rho^{s}+\rho^{\gamma}f(r),\quad \gamma=\frac{2(n-k)+(k-n+q){s}}{k}.
\]
The assumption $s\leq \frac{2n-2k}{n-q}$ ensures that $\gamma\geq s$. Since the Hessian $D^2 w$ at the point $x=\rho e_{n-k}+re_{n}$ takes the form
\[
\left[\begin{array}{cccc}
\frac{s\rho^{s-1}+\gamma \rho^{\gamma-1} f}{\rho}  \I_{n-k-1} & 0 & 0 & 0\\
0 & s(s-1)\rho^{s-2}+\gamma(\gamma -1)\rho^{\gamma-2}f  & 0 & \gamma \rho^{\gamma-1} r\\
0 & 0 & \rho^{\gamma} \I_{k-1}  & 0\\
0 & \gamma \rho^{\gamma-1}r  & 0 & \rho^{\gamma}
\end{array}\right],
\]
then $w$ is a convex function in $\left\{ r \leq \tau\right\}$ if $\tau>0$ is sufficiently small, and satisfies (see Lemma \ref{lem:determinat formula})
\[
\begin{split}
\det D^2 w
=\rho^{q{s}}\left({s}+\gamma \rho^{\gamma-{s}} f\right)^{n-k-1}
\left[ {s}({s}-1)+\gamma(\gamma -1)\rho^{\gamma-{s}}f 
- \gamma^2\rho^{\gamma-{s}} r^2\right].
\end{split}
\]
With this subsolution, the same argument as that for case (i) produces the desired solutions.
\end{proof}

Mooney \cite{mooney2015partial} established that every subsolution $w$ of \eqref{eq:mae classical} satisfies $\mathcal{H}^{n-1}(\Sigma_w) = 0$. He also proved the $W^{2,1}$ regularity for solutions $v$ of \eqref{eq:mae classical}, even in the case where $\Sigma_v$ contains uncountably many singular faces (Note that if the solutions of \eqref{eq:mae classical} are strictly convex, then they are $W^{2,1+\varepsilon}$ as shown in \cite{dephilippis2013w21regularity, dephilippis2013note,schmidt2013w21estimates}, which is the best possible due to the counterexample in \cite{wang1996regularity}). Mooney \cite{mooney2016counterexamples} further pointed out that the $W^{2,1}$ regularity may fail for supersolutions due to their degeneracy (see also \cite{dephilippis2024singular}). In light of the degeneracy for  \eqref{eq:obs eq q g=1}, we have the following analogous result:

\begin{Proposition}\label{prop:cylinder K}
Let $q > 0$. There exist a solution of \eqref{eq:obs eq q g=1} such that   $ K = \{|x'| \leq 1/2\}\cap B_{1}(0)$, $\Gamma_{nsc} = \{|x'| = 1/2\}\cap B_{1}(0)$ form an $(n-1)$-dimensional cylindrical hypersurface, where $x'=(x_1,\cdots,x_{n-1})$, and the solution is merely Lipschitz continuous and does not belong to $W^{2,1}$ in a neighborhood of $K$.
\end{Proposition}
 \begin{proof}
Let us consider the coordinates as in Lemma \ref{lem:determinat formula} with $k=1$.
Let 
\[
\rho=|x'|,\quad r=|x_{n}|,\quad f(r)=1+\frac{1}{2}r^2
\]
and
\[
d  :=d(\rho )= \max \left\{\rho-\frac{1}{2},0 \right\}  \quad \text{ so that } \operatorname{dist}(x,K)=d(|x'|) .
\]
Let us consider the following function
\[
w\left(x', x_n\right)=u(\rho,r)=d + d^{s}f(r), \quad 
s=\frac{q+2}{2}.
\]
The assumption $q>0$ ensures $s>1$. Since the Hessian $D^2 w$ at the point $x=\rho e_{n-1}+re_n,\rho >\frac{1}{2}$ takes the form
\[
\left[\begin{array}{ccc}
\frac{1+s d^{s-1}f}{\rho} \I_{n-2} & 0 & 0\\
0 & s(s-1) d^{s-2}f  & s d^{s-1}r\\
0 & s d^{s-1}  r  & d^{s}
\end{array}\right],
\]
then $w$ is a convex function in $\left\{ r \leq \tau\right\}$ if $\tau>0$ is sufficiently small, and satisfies (see Lemma \ref{lem:determinat formula})
\[
\det D^2 w
= \left(\frac{1+s d^{s-1}f }{\rho}\right)^{n-2} 
d^{2s-2}\left[s(s-1)  
- s^2 r^2\right] \geq cw^{q} .
\]
With this subsolution, the same argument as that for Theorem \ref{thm:dim optimize2} produces the desired solutions. Indeed, observing the Lipschitz growth of $v$ around $\{|x'|=\frac{1}{2}\}$,  we see that $\Delta v$ fails to be integrable. Consequently, $v \notin W^{2,1}_{loc}$ around $K$.
 \end{proof}

Next, we are going to use the examples in Theorem \ref{thm:dim optimize2} to construct new examples that prove Theorem \ref{thm:dim optimize}.

\begin{Proposition}\label{prop:sing on k skeleton}
Let $n+q >2$ and  $k=\left\lceil\frac{n+q}{2}\right\rceil - 1$. For each compact convex polytope $P \subset \R^n$ with $|P| > 0$,  there exist a domain $\Omega \supset P$ and a non-trivial solution $v$ to \eqref{eq:obs eq q g=1} in $\Omega $ satisfying
\[
(P\setminus\Gamma_0) \subset \{x \in \Omega:\; v(x) = 0\} , \quad \Gamma_{nsc} = \Gamma_{k} \cap \Omega.
\]
For $n \geq 3$ and $0<k <n-1$,  the free boundary $\Gamma$ is merely Lipschitz continuity at $\Gamma_{k}$, where $\Gamma_{j}$ denote  the $j$-skeleton\footnote{The $j$-skeleton of $P$, denoted by $\Gamma_j(P)$ or  $\Gamma_j$, is the union of all $j$-dimensional (exposed) faces of $P$.} of $P$.
\end{Proposition}

\begin{proof} 
Let us denote all $k$-dimensional (exposed) faces of $P$ as $E_1, \dots, E_m$. For each $E_i$, there exists a linear function $\ell_i$ such that
\[
E_i = P \cap \{\ell_i = 0\}, \quad P \subset \{\ell_i \leq 0\}.
\]
Furthermore, let us take a larger compact convex domain $\Omega$ satisfying
\[
\bigcap_{i=1}^m \{\ell_i < 0\} \subset \Omega, \quad \bigcap_{i=1}^m \{\ell_i \leq 0\} \cap \partial \Omega = \Gamma_0.
\]
Since $\Omega$ is compact, after an appropriate affine transformation, we may assume, without loss of generality, that $\Omega \subset B_{1/4}(0)$.

Let $\Phi$ denote the merely Lipschitz solution of \eqref{eq:obs eq q g=1} constructed in proof of Theorem \ref{thm:dim optimize2} (corresponding to the case $s=1$ there) such that $\{\Phi = 0\} = B_1(0) \cap \mathbb{R}^k$ and $\Phi$ is of linear growth at $\{\Phi = 0\}$. 
For each $1\leq i\leq m$,
by applying a rotation and translation to $\Phi$, we can relocate its zero set to the plane containing $E_i$, through which we obtain a function defined on $B_{1/2}(0)$, denoted as $\Phi_i$. 
Making use of the linear growth of $\Phi_i$ around its coincidence set, by taking $M_1$ sufficiently large, we obtain 
\[
\det D^2 (\Phi_i+M_1\ell_i)= \Phi_i^q \chi_{\{\Phi_i>0\}}\geq c_{P,M_1} \max\{\Phi_i+M_1\ell_i,0\}^q  ,
\]
and
\[
P\subset \{\Phi_i+M_1\ell_i \leq 0\} \subset \{ \ell_i \leq 0\}.
\]
Since $q < n$, for sufficiently large $M_2$, the nonnegative convex function 
\[
w =M_2\max\left\{ \Phi_1+M_1\ell_1,\cdots, \Phi_m+M_1\ell_m, 0\right\}.
\]
is a subsolution of \eqref{eq:obs eq q g=1} satisfying $\det  D^2 w\ge w^q\chi_{\{ w > 0\}}$ in $\Omega$ with $w=0$ on $P$.

Let $v\ge 0$ be the solution to
\[
\det  D^2 v=v^q\chi_{\{ v > 0\}} \quad \text{in } \Omega , \quad v = w \quad \text{on } \partial \Omega .
\] 
Then $v=w=0$ on $\Gamma_0 \subset \partial\Omega$. Since $v\ge0$ is convex and $\Gamma_0$ consists of the vertices of $P$, it follows that $v\equiv0$ on $P$.
Also, by the comparison principle, we have $v \ge w$. Together with the inclusions $\Gamma_k \subset P \subset \{v = 0\}$ and that $w$ is only Lipschitz continuous (and no better) at $\Gamma_k$, we conclude that $v$ itself is merely Lipschitz continuous near $\Gamma_k$.
Let  $K:=\left\{ v=0\right\}$, and let 
\[
F:= \bigcap_{i=1}^m \{\Phi_i+M_1\ell_i \leq 0\} .
\]
Then, 
\begin{equation}\label{eq:subset}
P \subset K \subset \left\{ w=0\right\} \subset F \subset \bigcap_{i=1}^m \{\ell_i \leq 0\} \subset   \overline\Omega.
\end{equation}
Since $\bigcap_{i=1}^m \{\ell_i \leq 0\} \cap \partial\Omega = \Gamma_0$, the relation \eqref{eq:subset} further shows that each $E_i$ is an exposed face of $K$. It then follows from \eqref{eq:definition of gammansc} that
\[
\Gamma_k = \bigcup_{i=1}^m E_i \subset \Gamma_{nsc}.
\]
On the other hand, since $\partial K\cap\partial\Omega=\Gamma_0$, Proposition \ref{thm:nsc set} and Theorem \ref{thm:dim of E} imply that $\Gamma_{nsc} \subset \Gamma_{k} $. In conclusion,
\[
\Gamma_{nsc}=\Gamma_{k} .
\]

When $n\ge 3$ and $0<k < n-1$, the boundary $\partial F$ is merely Lipschitz continuous at each $E_i$. Consequently, it follows from \eqref{eq:subset} that $\partial K$ is also merely Lipschitz continuous at each $E_i$.
\end{proof}

\begin{Remark}
From the construction in the proof of Proposition \ref{prop:sing on k skeleton}, we can further  see that 
\begin{itemize}
\item For $n \geq \max(3,q+2)$, then $k<n-1$, and therefore, $P\subsetneq K$. This means that the free boundary $\Gamma$ simultaneously contains both $\Gamma_{sc}$ and $\Gamma_{nsc}$.
\item For $q > n-2$, the maximal dimensional case $k = n-1$ occurs, resulting in $K = P$ (the entire polytope). This means that $\Gamma=\Gamma_{nsc}$.
\end{itemize} 
\end{Remark}

\begin{proof}[Proof of Theorem \ref{thm:dim optimize}]
It follows from Proposition \ref{prop:sing on k skeleton}.
\end{proof}

\section{Strict inclusion and stability of coincidence sets}\label{sec:smp}
In this section, we investigate the strong maximum principle and a stability property for the Monge-Amp\`ere obstacle problem \eqref{eq:obs problem q}. In this section, we always assume that $g(x)$ is a bounded and strictly positive function in $\Omega$.

\subsection{Strict inclusion of coincidence sets}
In \cite{jian2025strong},  Jian and the second author studied the strong maximum principle for the classical Monge-Amp\`ere equation \eqref{eq:mae classical}, 
demonstrating that $\Sigma_w$, the set of non-strictly convex points, constitutes the fundamental obstruction to the local validity of this principle.

\begin{Theorem}[Theorem 1.2 in \cite{jian2025strong}]\label{thm:smp nece}
Let $v$ be a convex solution to \eqref{eq:mae classical}. For any line segment $L \subset \Sigma_v$ and any point $x_0 \in \mathring{L}$, there exists a different convex solution $w \not\equiv v$ of \eqref{eq:mae classical} defined near $x_0$ that touches $v$ from below at $x_0$.
\end{Theorem}

The next result is a direct consequence of Theorem 1.5 in \cite{jian2025strong}.
\begin{Theorem}\label{thm:smp 1}
Suppose $v_1 \leq v_2$ are solutions to \eqref{eq:obs problem q}. Then on each connected component of $\{v_2>0\}$, we have either $v_1  \equiv v_2$ or $\{v_1 =v_2\} \subset (\Sigma_{v_1} \cap \Sigma_{v_2})$
\end{Theorem}

We extend the preceding analysis to the obstacle problem \eqref{eq:obs problem q}, with particular emphasis on strict inclusion relationships between coincidence sets. Specifically, we investigate the following problem:
\begin{question}\label{ques:smp for obstacle}
Let $v_1 \leq v_2$ be solutions to \eqref{eq:obs problem q}, and denote $K_i = \{v_i = 0\}$ for $i = 1,2$. If $v_1 \not\equiv v_2$ and $K_1 \neq \emptyset$, under what conditions can we conclude that $K_2 \subset \subset K_1$ in $\Omega$ and $v_1 < v_2$ in $\Omega \setminus K_2$?
\end{question}
 
The following two propositions address the above equation.
 
\begin{Proposition}[Failure of the strong maximum principle at $\Sigma_v$]\label{prop:non-strict decay}
Let $v$ be a solution to \eqref{eq:obs problem q}. For any convex subset $E \subset \Sigma_v$ on which $v$ is linear such that $(\overline E)^{ext} \subset \partial \Omega$, there exists a family of convex solutions $\{v_t\}_{t>0}$ to \eqref{eq:obs problem q} satisfying:
\begin{enumerate}
    \item[(i)] $v_t \to v$ uniformly as $t \to 0^+$,
    \item[(ii)] $v_t \geq v$ with $v_t \not\equiv v$,
    \item[(iii)] $v_t \equiv v$ on $E$,
\end{enumerate}
where $\Sigma_v$ is defined in Proposition \ref{thm:nsc set}. When $E \subset \partial K$, we also have $ E \subset \partial\{v_t = 0\}$.
\end{Proposition}
\begin{proof}
Let $v_t$ solve the obstacle problem \eqref{eq:obs problem q} with boundary data $v_t = v + t\operatorname{dist}(x,E)$ on $\partial\Omega$. Clearly, $v_t \not\equiv v$. The comparison principle implies that $v_{t} \geq v$ in $\Omega$ and $\|v_t - v\|_{L^\infty(\Omega)} \leq Ct$. 
By convexity, since $(\overline E)^{ext} \subset \partial \Omega$ and $v_t=v$ on $(\overline E)^{ext}$, we obtain that $v_{t} = v$ on $E$.
Thus, $E \subset \Sigma_{v_{t}}$.
\end{proof}

\begin{Proposition}[Validity of the strong maximum principle at $\Gamma_{sc}$]\label{prop:strict decay}
Let $v_1 \leq v_2$ be solutions to \eqref{eq:obs problem q}. Assuming either $q = 0$ or $g \in Lip(\Omega)$. Then on each connected component of $\Omega \setminus \mathring{K_2}$, we have either $v_1  \equiv v_2$ or $\{v_1 =v_2\} \subset (\Sigma_{v_1} \cap \Sigma_{v_2})$.
\end{Proposition}

Proposition \ref{prop:strict decay} follows from combining Theorem \ref{thm:smp 1} with Lemma \ref{lem:strict-decay} below. 
To avoid confusions, we use $ K_i, \Gamma^{v_i}, \Gamma_{nsc}^{v_i}, \Gamma_{sc}^{v_i}$ to denote the $K, \Gamma, \Gamma_{nsc}, \Gamma_{sc}$ for $v_i$, respectively.

\begin{Lemma}\label{lem:strict-decay}
Let $v_1 \leq v_2$ be solutions to \eqref{eq:obs problem q}. Assuming either $q = 0$ or $g \in Lip(\Omega)$, then for any $x_0 \in \partial K_1 \cap \partial K_2$, either $x_0 \in \Gamma_{nsc}^{v_1} \cap \Gamma_{nsc}^{v_2}$ or $v_1 \equiv v_2$ in a neighborhood of $x_0$.
\end{Lemma}

\begin{proof}
For simplicity, assume $K_2 \subset K_1 \subset \{x_n \geq 0\}$ with $x_0 = 0\not\in \Gamma_{nsc}^{v_1} \cap \Gamma_{nsc}^{v_2}$ and $g(0) = 1$.  Therefore, $0 \in \Gamma_{sc}^{v_2}$ (since if $0\in \Gamma^{v_1}_{sc}$ then $0\in \Gamma^{v_2}_{sc}$ due to $K_2\subset K_1$), and it suffices to prove $v_1 = v_2$ near $0$. We  further assume that $B_\rho(0)\subset \Omega$ and $B_\rho(0)\cap \Sigma_{v_2}=\emptyset$ for some $\rho > 0$. 

Suppose by contradiction that $v_2 \not\equiv v_1$. Then Theorem \ref{thm:smp 1} yields $v_2 > v_1$ in $B_\rho(0) \setminus K_2$.

{\bf (i): The case of $|K_2|>0$.} 
After applying the normalization in Definition 2.21 of \cite{jin2025regularity}, which we still denote by $v_2$, we may assume that for some $C>1$,   $B_C(0)\setminus K_2$ is connected and $\Sigma_{v_2} \cap B_C(0)=\emptyset$. 

If $q=0$, we have $v_2 \geq v_1 - \delta x_n$ on  $\partial (B_1(0)\cap\{x_n\le 0\})$ holds for all sufficiently small $\delta > 0$.
The comparison principle then implies $v_2(-te_n) \geq \delta t$ for small $t>0$. 
However, this contradicts the $C^1$ regularity of $v_2$ at $0$, where $\nabla v_2(0) = 0$. 

If $q \in (0,n)$,  after the above normalization, and we can further assume that
\[
\| D g\|_{L^\infty(B_C(0))} \leq \frac{1}{4}.
\]
Noting that $v_2 > v_1$ outside $K_2$, the strict convexity of $\partial K_2$ implies that:
\[
v_2  > v_1 \quad \text{ on } \partial B_1(0)\cap \{ x_n \leq \tau\}  
\]
provided $\tau$ is sufficiently small. 
Let  $\varepsilon > 0$ be a small constant. Consider the auxiliary function  
\begin{equation}\label{eq:auxiliary w}
w (x) = v_2\left( \T x  \right)\quad \text{in }   B_1(0),
\end{equation}
where 
\[ 
\T x = x -\varepsilon(x_{n} - \tau)e_{n}.
\]
We observe that
$
w  > v_1 \text{ on } \partial B_1(0)\cap \{ x_n \leq \tau\}  
$
provided $\varepsilon$ is small by continuity, and
$
 w = v_2 \geq v_1  \text{ on } B_1(0)\cap  \{ x_n = \tau\}.
$
In conclusion, $w \geq v_1 $ on $\partial (B_1(0)\cap  \{ x_n = \tau\})$.  Moreover, the function $w$ satisfies
\[
\det  D^2 w =(1-\varepsilon)^2   (g\circ \T) w^q \leq
(1-\varepsilon)^2(1+\varepsilon)   g  w^q \leq  gw^q .
\] 
Applying the comparison principle, we obtain  $w\geq v_1  $ in $B_1(0)\cap  \{ x_n\le \tau\}$. This leads to a contradiction that
\begin{equation}\label{eq:contradiction 1}
-\frac{\tau\varepsilon}{1-\varepsilon}e_n \subset K_1 \subset \{x_n \geq 0\}.
\end{equation}

{\bf (ii): The case of $|K_2|=0$.} 
In this case, we must have $K_2=\{0\}$. 
When $q=0$, the proof follows from the strong maximum principle, see Theorem 1.3 in \cite{jian2025strong}.
When $q \in (0,n)$, 
we apply the normalization (from \cite[Page 25]{jin2025regularity}): for any $h>0$ small, choose affine transformations $\D_h$ with $\det \D_h=h^{\frac{n-q}{2}}$ such that $B_c(0)\subset S_{h}^{v_{2}}\subset B_C(0)$, where $S_h^v := \{x \in \Omega : v(x) < h\}$, and define
\begin{equation}\label{eq:normalized}
v_{2,h}(x) = \frac{v_2(\D_h x)}{h},
\end{equation}
which satisfies
\[
\det D^2 v_{2,h}=g_hv_{2,h}^q\chi_{\{v_{2,h}>0\}}
\]
with $g_h =g\circ \D_h$. By choosing sufficiently small $h$, we may assume that $\Sigma_{v_{2,h}} \cap B_4(0)=\emptyset$, and $\| D g_h\|_{L^\infty(B_4(0))}\leq  1$. 
Repeating the argument from the case $|K_2|>0$, the auxiliary function $w_h:=v_{2,h}\left( \T x  \right)$ in \eqref{eq:auxiliary w} satisfies $w_h\geq v_1$ in $B_{1}(0)$, leading to the contradiction \eqref{eq:contradiction 1}.
\end{proof}

When $|K_2|=0$, the Lipschitz condition $g \in Lip(\Omega)$ in Lemma \ref{lem:strict-decay} becomes unnecessary, as shown below. 
\begin{Lemma}\label{lem:strict-decay K=0}
Let $v_1 \leq v_2$ be solutions to \eqref{eq:obs problem q}.  If $0\in \partial K_1$ and $K_2=\{0\}$, then $v_1\equiv v_2$.
\end{Lemma}
\begin{proof}  
Let us assume that $K_1 \subset \{x_n\geq 0\}$. For $h>0$ small, let
\[
a_h:=a_h(v_1,v_2) = \inf_{x\in \partial S_h^{v_2}} \frac{v_2(x)}{v_1(x)}  \geq 1.
\]
Since $a_h \geq 1$ and $n>q$, the function $a_hv_1 $ is a subsolution to \eqref{eq:obs problem q}, i.e., 
$\det D^2 (a_h v_1)\ge (a_h v_1)^q\chi_{\{a_hv_1>0\}}$. By the comparison principle, $v_2  \geq a_hv_1 $ in $S_h^{v_2}$, which implies that $a_h$ is non-increasing in $h$.

By Theorem 1.7 in \cite{jin2025regularity}, $v_2$ is $C^{1,\alpha}$ and strictly convex near $0$. This implies   $ \Sigma_{v_2}\cap S_{3h}^{v_2}=\emptyset$ for small $h>0$. 

(i):  Suppose $a_{h_0} = 1$ for a small $h_0>0$. Then the strong maximum principle (see Theorem 1.5 of \cite{jian2025strong})  implies that  
$v_2 \equiv  v_1$ in $\{v_2>0\}$. By continuity, this identity extends globally in $\Omega$, completing the proof.

(ii): Suppose $a_h > 1$ for all small $h>0$.
Then the strong maximum principle (see Theorems 1.5 and 1.6 of \cite{jian2025strong}) implies that in $S_h^{v_2}\setminus\{0\}$, either $v_2 > a_h v_1$ or $v_2 \equiv a_h v_1$. Nonetheless, if $v_2 \equiv a_h v_1$, then it follows from the equations of $v_1$ and $v_2$ that $a_h = 1$. Hence, $v_2 > a_h v_1$ holds in $S_h^{v_2}\setminus\{0\}$. Therefore, $a_h$ is strictly decreasing in $h$.

The positive constants $c$ and $C$ below may vary from line to line but depend only on $n,q,\lambda$ and $\Lambda$.

Next, we establish the uniform upper bound $a_h \leq C$. Arguing by contradiction, suppose for every $\varepsilon > 0$, there exists $h_0 > 0$ such that $a_{h_0} \geq \varepsilon^{-1}$. Then we have $S_{\varepsilon h_0}^{v_1} \supset S_{h_0}^{v_2}$. Since we assumed that $K_1 \subset \{x_n\geq 0\}$, then Lemma \ref{lem:volume concidence} yields 
\[
|S_{h_0}^{v_2} \cap \{x_n <0\} | \leq |S_{\varepsilon h_0}^{v_1} \cap \{x_n <0\} | \leq C(n,q) (\varepsilon h_0)^{\frac{n-q}{2}}. 
\]
For sufficiently small $\varepsilon > 0$, this leads to a contradiction with Lemma 2.13 of \cite{jin2025regularity}, which states that the section $S_{h_0}^{v_2}$ is balanced about  $0$ and satisfies
\[
c(h_0)^{\frac{n-q}{2}} \leq |S_{h_0}^{v_2}| \leq C(h_0)^{\frac{n-q}{2}}.
\] 
This proves that there exists $C>0$ such that $a_h\le C$ for all $h>0$ small.

Let us define 
\[
\tilde{a}= \lim_{h\to 0^+}a_h \in (1,C].
\]

Apply the normalization procedure as in \eqref{eq:normalized} to $v_2$ and define  $v_{1,h}(x) = v_1(\D_h x)/h$ as well. Then we have $\det D^2 v_{1,h}=g_hv_{1,h}^q\chi_{\{v_{1,h}>0\}} $, and 
\[
a_s(v_{1,h},v_{2,h}) =\inf_{x\in \partial S_1^{v_{2,sh}}} \frac{v_{2,sh}(x)}{v_{1,sh}(x)} = \inf_{x\in \partial S_{sh}^{v_2}} \frac{v_2(x)}{v_1(x)}  =a_{sh}(v_1,v_2),\quad \forall s\in (0,1).
\] 
After passing to a subsequence of $h \to 0$,  
we can assume $g_h \to \tilde{g}$  in the measure sense with $\lambda < \tilde{g} < \Lambda$, and $v_{1,h}$ and $v_{2,h}$ converge locally uniformly to some nonnegative convex functions $\tilde{v}_1$ and $\tilde{v}_2$, respectively, that solves  
\[
\det D^2  \tilde{v}_i = \tilde{g} \tilde{v}_i^{q} \chi_{\{\tilde{v}_i > 0\}},\quad i=1,2.
\]
From the uniform convergence, we have $B_c(0) \subset S_{1}^{\tilde{v}_2}\subset B_C(0)$,  
and 
\[
\tilde{a}_s:=a_s(\tilde{v}_1,\tilde{v}_2)= \lim_{h\to 0^+} a_s(v_{1,h},v_{2,h}) = \lim_{h\to 0^+}a_{sh}(v_1,v_2)  =\tilde{a} >1,\quad \forall s\in  (0,1). 
\]
This is a contradiction, since $\tilde a_s$ is strictly decreasing in $s$ unless $\tilde a_s \equiv 1$, as proved at the begnning. This completes the proof.
\end{proof}

\begin{proof}[Proof of Proposition \ref{prop:strict decay}]
It follows from Theorem \ref{thm:smp 1} and Lemma \ref{lem:strict-decay}.
\end{proof}

\subsection{Stability of coincidence sets}

Finally, we investigate the stability of coincidence sets under perturbations of solutions. 

\begin{Definition}
For a solution $v$ of \eqref{eq:obs problem q}, we say that $K=\{x\in \Omega:\; v(x)=0\}$ is stable if for every sequence of solutions $\{v_i\}$ to  \eqref{eq:obs problem q} that converges locally uniformly to $v$, the coincidence sets $K_i := \{x\in\Omega: v_i (x)= 0\}$ satisfy $K_i \to K$ in the set-theoretic sense.
\end{Definition}


\begin{Proposition}\label{prop:stability of K}
Let $v$ be a solution to  \eqref{eq:obs problem q} with $K\neq \emptyset$. Then $K$ is stable if and only if $|K|>0$.
\end{Proposition}

\begin{proof} 
(i): Suppose $|K| =0$.  Then  either $K$ is a singleton, or $K=\Gamma_{nsc}$ containing at least two distinct points.

If $K$ is a singleton, say $K=\{0\} \subset \Omega$. Let $v_t$ solve \eqref{eq:obs problem q} with $v_t=v+t$ on $\partial\Omega$ for $t>0$. 
As in Lemma \ref{prop:non-strict decay},  we have $v_t$ locally uniformly converging to $v$ in $\Omega$ with $v_t \geq v$ and $K_t \subset \{0\}$, where $K_t = \{x \in \Omega:\;v_t(x) = 0\}$. Lemma \ref{lem:strict-decay K=0} excludes the possibility $K_t = \{0\}$. Thus, $K_t = \emptyset$, and therefore, $K$ is unstable. 

If $K=\Gamma_{nsc}$ contains at least two points, we let $x_0 \in \partial \Omega$ be an extremal point of $\overline{K}$.
Let $v_{t}$ solve  \eqref{eq:obs problem q} with $v_{t} =   v + t|x-x_0|$ on $\partial \Omega$. 
As in  Lemma \ref{prop:non-strict decay}, we have $v_t$ locally uniformly converging to $v$,  $v_{t} \geq v$ in $\Omega$, and $K_t \subset K$.
Since $v_t(x_0)=0$, $K_t$ is not a singleton. Proposition \ref{thm:nsc set} and Remark \ref{rem:dichotomy for Ktype}, together with $v_t>0$ on $\partial\Omega\setminus\{x_0\}$, imply $K_t = \emptyset$, and thus $K$ is unstable.

(ii): Suppose $|K| >0$.
Let us show that $K$ is stable.  The argument originates from Lemma 2.20 in our previous work \cite{jin2025regularity}, and we reproduce it here for completeness. 

Let $\{v_i\}_{i\geq 1}$ be a sequence of solutions  of \eqref{eq:obs problem q} locally uniformly converging to $v$. 
 By the Blaschke's selection theorem, we can always extract a subsequence (still indexed by $i = 1, 2, \dots$) such that the convex coincidence sets $K_i := \{v_i = 0\}$ converge to a closed (relative to $\Omega$) and convex set $K_\infty$.

The inclusion $K_\infty \subset K$ 
follows immediately from the uniform convergence, and it suffices to establish that $K_\infty =K$.
If $ K_{\infty} \subsetneq K$, then  we can find an open ball $B_{r_0}(x_0) \subset K \setminus K_{\infty}$. Moreover, we may assume that  $x_0 \notin K_i$ for all $i$ large. Then, due to convexity, there exist half spaces $H_i^-$ containing  $x_0$ such that $H_i^- \cap K_{i} =\emptyset$. Thus, $v_{i} > 0 $ on $ B_{r_0}(x_0)  \cap H_i^-$. Applying Lemma \ref{lem:volume concidence}, we obtain that $\sup_{x \in B_{r_0}(x_0)  \cap H_i^-  } v_{i}\geq c(n,q,\lambda,r_0) >0$. This contradicts to the fact that $v_{i}$ uniformly converges to $v =0$ on $B_{r_0}(x_0) \subset K$.
\end{proof}

\small

\bibliography{references}
\smallskip

\noindent T. Jin

\noindent Department of Mathematics, The Hong Kong University of Science and Technology\\
Clear Water Bay, Kowloon, Hong Kong\\
Email: \textsf{tianlingjin@ust.hk}

\medskip

\noindent X. Tu

\noindent Department of Mathematics, The Hong Kong University of Science and Technology\\
Clear Water Bay, Kowloon, Hong Kong\\[1mm]
Email:  \textsf{maxstu@ust.hk}

\medskip

\noindent J. Xiong

\noindent School of Mathematical Sciences, Laboratory of Mathematics and Complex Systems, MOE\\ Beijing Normal University, 
Beijing 100875, China\\
Email: \textsf{jx@bnu.edu.cn}

\end{document}